\documentclass[Journal]{IEEEtran}
\usepackage{amsmath}
\usepackage{amsthm}
\usepackage{algorithm}
\usepackage{algpseudocode}
\usepackage{graphics}
\usepackage{epsfig}
\usepackage{mathrsfs}
\usepackage{amssymb}
\usepackage[OT1]{fontenc}
\usepackage[usenames,dvipsnames,svgnames,table]{xcolor}
\usepackage{multicol}
\usepackage{epstopdf}

\begin{document}
\newtheorem{emp}{Example}
\newtheorem{cn}{Conjecture}
\newtheorem{cond}{Condition}
\newtheorem{lm}{Lemma}
\newtheorem{thm}{Theorem}
\newtheorem{cor}{Corollary}
\newtheorem{df}{Definition}
\newtheorem{pf}{Proof}
\newtheorem{alg}{Algorithm}
\newtheorem{rmk}{\bf Remark}


\title{From Co-prime to the Diophantine Equation Based Sparse Sensing}
\hyphenation{IEEE Signal Processing Letters}


\let\lc\langle
\let\rc\rangle

\date{}
\date{}

\author{Hanshen Xiao and Guoqiang Xiao

\thanks{Hanshen Xiao is with CSAIL and EECS Department of MIT, Cambridge, USA. E-mail: hsxiao@mit.edu.}
\thanks{Guoqiang Xiao is with College of Computer and Information Science, Southwest University, Chongqing, China. E-mail: gqxiao@swu.edu.cn.}}
\maketitle

\maketitle

\begin{abstract}
With a careful design of sample spacings either in temporal and spatial domain, co-prime sensing can reconstruct the autocorrelation at a significantly denser set of points based on Bazout theorem. However, still restricted from Bazout theorem, it is required $O(M_1+M_2)$ samples to estimate frequencies in the case of co-prime sampling, where $M_1$ and $M_2$ are co-prime down-sampling rates. Besides, for Direction-of-arrival (DOA) estimation, the sensors can not be arbitrarily sparse in co-prime arrays. In this letter, we restrain our focus on complex waveforms and present a framework under multiple samplers/sensors for both frequency and DOA estimation based on Diophantine equation, which is essentially to estimate the autocorrelation with higher order statistics instead of the second order one. We prove that, given arbitrarily high down-sampling rates, there exist sampling schemes with samples to estimate autocorrelation only proportional to the sum of degrees of freedom (DOF) and the number of snapshots required. In the scenario of DOA estimation, we show there exist arrays of $N$ sensors with $O(N^3)$ DOF and $O(N)$ minimal distance between sensors.

\end{abstract}

\begin{section}{Introduction}
\noindent The study on co-prime sparse sensing, which was initialized in \cite{coprime1,coprime2}, spans almost last decade and has witnessed tremendous progress \cite{co-array1,pushing,qin,localization,co-array2}. It can be used in autocorrelation reconstruction based estimation with enhanced degrees of freedom (DOF) while it also preserves sparsity either in time or space domain. The key idea of co-prime sensing is that, for two sequences $\mathscr{M}_1=\{m_1M_1\tau, m_1=0,1,2,...,2M_2-1\}$ and $\mathscr{M}_2=\{m_2M_2\tau, m_2=0,1,2,...,M_1-1\}$, where $M_1$ and $M_2$ are co-prime integers and $\tau$ is some unit, the difference set of the pair of elements from $\mathscr{M}_1$ and $\mathscr{M}_2$ respectively will include all consecutive multiples of $\tau$ starting from $-M_1M_2\tau$ to $M_1M_2\tau$. In the applications of frequency estimation, the two sequences, $\mathscr{M}_1$ and $\mathscr{M}_2$, represent the sampling instants in the time domain from two samplers, while for the case of Direction-of-arrival (DOA) estimation, $\mathscr{M}_1$ and $\mathscr{M}_2$ correspond to the sensor positions of two uniform arrays, respectively. Such carefully designed difference set from $\mathscr{M}_1$ and $\mathscr{M}_2$ can construct $O(M_1M_2)$ consecutive lags for autocorrelation estimation, which is usually used in spectral estimation methods, such as MUSIC and ESPRIT \cite{music} etc. On the other hand, $M_1$ and $M_2$ can be arbitrarily large provided the co-prime property is preserved. Without dispute, the enhanced sparsity and DOF will make breakthrough in physical limitations. Since such idea can be generalized to any integer Ring, apart from one dimension, higher dimension cases have also been explored \cite{twodimension-1,twodimension-2,multidimension}. Though co-prime sensing has been well studied, two problems related remain open. 
\begin{enumerate}
\item Whether in the case of frequency estimation, there exists a more flexible sampling scheme to trade off time delay against resolution. If $L$ snapshots are required for autocorrelation estimation, the time delay of co-prime sampling is $O(LM_1M_2)$. When $M_1$ and $M_2$ are large enough, such delay is unacceptable.
\item Whether there exists an array structure which can achieve a larger minimal distance between sensors and more enhanced DOF. In DOA estimation, a closer placement of sensors will incur a tighter mutual coupling. However, the minimal distance between sensors in existing co-prime arrays or (super) nested arrays  \cite{2014nested,super1,super2,super3,nest,2019nested}, is fixed to be a half of wavelength.
\end{enumerate}
In this paper, we answer both questions affirmatively in the case of complex waveforms. {In Section II, details of Diophantine equation based sampling are provided and Theorem 1 suggests the new scheme only requires $O(K+L)$ samples, where $K$ is the number of consecutive lags required. Further, the case of arbitrarily $N$ samplers is investigated and the results are shown in Theorem 2. In Section III, a framework of DOA estimation is presented with similar idea. Theorem 3 shows Diophantine equation based arrays can provide $O(N^3)$ DOF and the minimal distance between sensors is $O(N)$. }
\end{section}

\section{Diophantine Equation Based Sampling}
\noindent We first flesh out the main idea of co-prime sampling. Let us consider a complex waveform formed by $D$ sources,
\begin{equation}
x(t) = \sum_{i=1}^{D} A_{i}e^{j(\omega_{i}t+\phi_{i})} = \sum_{i=1}^{D} A_{i}e^{j\phi_{i}}e^{j\omega_{i}t}
\end{equation}
where $\omega_{i} = 2\pi f_{i}T_{s}$ is the digital frequency and $T_s$ is the Nyquist sampling interval. $A_{i}$, $f_i$ and $\phi_{i}$ are the amplitude, the frequency and the phase of the $i$th source, respectively. The phases $\phi_{i}$ are assumed to be random variables uniformly distributed in the interval $[0,~2\pi]$ and uncorrelated from each other. The two sampled sequences with sampling interval $M_1T_s$ and $M_2T_s$, respectively, are expressed as
\begin{equation}
\begin{aligned}
 &x_1[n] = x(nM_1T_s)+w_1(n)\\
 &x_2[n] = x(nM_2T_s)+w_2(n)
\end{aligned}
\end{equation}
where $w_1(n)$ and $w_2(n)$ are zero mean Gaussian white noise with power $\sigma^2$. 
In \cite{coprime1,coprime2}, it is shown that for $r \in \mathbb{Z}$, there exist $m^r_{1k} \in \{rM_2, rM_2+1, ... ,(r+2)M_2-1\}$ and ${m}^r_{2k} \in \{rM_1, rM_1+1, ... ,(r+1)M_1-1\}$ such that
\begin{equation}
\label{Bzout}
{m}^r_{1k}M_1- {m}^r_{2k}M_2 = k
\end{equation}
for $k = \{-M_1M_2,...,0,1, ...,M_1M_2\}$. Furthermore, the autocorrelation of $x[n]$ can be expressed as
\begin{equation}
\label{corre}
R_x[k] = \mathbb{E}_n\{x[n]x^*[n-k]\}=\sum_{i=1}^D A^2_i e^{j\omega_{i}k}
\end{equation}
According to equation (\ref{Bzout}), $R_x[k]$ can be also estimated by using the average of the inner product of those pairs $\{x_1[m^r_{1k}], x_2[m^r_{2k}]\}$ with sparse sampling. To this end, by rewriting $R_x[k]$ in the context of $x_1[m^r_{1k}]$ and $x_2[m^r_{2k}]$, i.e., $E_r\{x_1[m^r_{1k}] \cdot x^*_2[m^r_{2k}]\}$, we have demonstrated that $\mathbb{E}_n\{x[n]x^*[n-k]\}=E_r\{x_1[m^r_{1k}] \cdot x^*_2[m^r_{2k}]\}$ in (\ref{haha}), which implies the validity of co-prime sampling.
\newcounter{mytempeqncnt}
\begin{figure*}[!h]
\normalsize
\setcounter{mytempeqncnt}{\value{equation}}
\setcounter{equation}{5}
  \begin{equation}
  \label{haha}
\begin{aligned}
&\mathbb{E}_r[x_1[{m}^r_{1k}]x^*_2[{m}^r_{2k}]] =\mathbb{E}_r \{[ \sum_{i=1}^{D} A_ie^{j\phi_i}e^{j\omega_{i}{m}^r_{1k}M}+w_1({m}^r_{1k})][ \sum_{i=1}^{D} A_ie^{j\phi_i}e^{-j\omega_{i}{m}^r_{2k}N}+w^*_2({m}^r_{2k}) ]\}\\
& = \mathbb{E}_r [ ~\sum_{i=1}^{D} A^2_i e^{j\omega_{i}({m}^r_{1k}M-{m}^r_{2k}N)}~]+\mathbb{E}_r[~\sum_{i \not = l}A_ie^{j\phi_i} e^{j\omega_{i}{m}^r_{1k}M} A_le^{j\phi_l}e^{-j\omega_{l}{m}^r_{2k}N}~]+ \mathbb{E}_r [w_1({m}^r_{1k})w^*_2({m}^r_{2k})]\\
& = \mathbb{E}_r [ ~\sum_{i=1}^{D} A^2_i e^{j\omega_{i}k}] +  \mathbb{E}_r[~\sum_{i \not = l} A_i A_l e^{j(\phi_i-\phi_l)} e^{j\omega_{l}k}  e^{j{m}^r_{1k}M( \omega_i -\omega_{l})}~]  = \sum_{i=1}^{D} A^2_i e^{j\omega_{i}k} = \mathbb{E}[x[n]x^*[n-k]]
\end{aligned}
\end{equation}
\setcounter{equation}{\value{mytempeqncnt}}
\hrulefill
\end{figure*}
\setcounter{equation}{6}
If $L$ snapshots are used to estimate each $\mathbb{E}_n \{x[n]x^*[n-k]\}$, the delay of co-prime sampling is $O(LM_1M_2T_s)$ \cite{complexity}. When $M_1$ and $M_2$ are large enough, such delay is unacceptable in many applications.

{When very high resolution is not necessary, it would be beneficial to balance the time delay against the resolution in order to reduce the delay. To this end, we propose to estimate the autocorrelation with higher order statistics. Clearly higher order statistics will be less robust than the lower ones, while, as shown soon, much stronger flexibility is achievable. To be specific, we consider a class of generic Diophantine equation
\begin{equation}
\label{threedio}
m_1M_1 + m_2M_2 + m_3M_3 = k,
\end{equation}
instead of the special case with only two variables in (\ref{Bzout}). Equation (\ref{threedio}) yields the seed of proposed framework of Diophantine equation based sparse sensing and it has integer solutions if and only if $gcd(M_1,M_2,M_3)|k$ is satisfied. Here $gcd$ denotes the greatest common divisor} \footnote{Please note we do not require those three integers to be relatively co-prime.}.

{Suppose there are three samplers with down sampling rate $M_1$, $M_2$ and $M_3$, respectively, where $gcd(M_1,M_2,M_3)=1$. In order to find out the solutions, $\{m_1, m_2, m_3\}$, of equation (\ref{threedio}), we first construct equation (\ref{Dio}). There exist two groups of integers, $\{a_1, a_2, a_3\}$ and $\{b_1,b_2,b_3\}$, such that
\begin{equation}
\label{Dio}
\left\{
            \begin{array}{lr}
a_1M_1+a_2M_2+a_3M_3 = 0 \\
b_1M_1+b_2M_2+b_3M_3 = 1 \\
 \end{array}
             \right.
\end{equation}
and the signs of $\{a_i\}$ are not the same, nor are $\{b_i\}$. Here we use the fact that \cite{hardy} (\ref{Dio}) has solutions if $gcd(M_1,M_2,M_3) = 1$. Then according to equation (\ref{Dio}) for each $k, l \in \mathbb{Z}$, $k=1,2,...,K$ and $l=1,2,...,L$, we have
\begin{equation}
\label{diosolution}
(kb_1+la_1)M_1 + (kb_2+la_2)M_2 +(kb_3+la_3)M_3 = k
\end{equation}
}
Without loss of generality, supposing $(kb_1+la_1)$ and $(kb_3+la_3)$ to be positive and $(kb_2+la_2)$ to be negative, we construct the sequence $\{{x}_1[kb_1+la_1]{x}^*_2[-(kb_2+la_2)]{x}_3[kb_3+la_3]\}$ to estimate the autocorrelation, where ${x}_i[n] = x(M_inT_s) + w_i(n)$ denotes the sample sequence with a downsampling rate $M_i$. In (\ref{proof-main}), we show that $\mathbb{E}_l \{{x}_1[kb_1+la_1]{x}^*_2[-(kb_2+la_2)]{x}_3[kb_3+la_3]\} = \sum_{i=1}^D A^3_ie^{j\phi_i}e^{j\omega_ik}$, once $a_1M_1(\omega_i-\omega_v)+a_2M_2(\omega_u-\omega_v) \not = 0$ for $i\not =v \not = u \in \{1,2,...,D\}$, which fails with negligible probability. The results of autocorrelation estimated using the constructed sequence will be the same as that of $x[n]$ with denser samples.
\begin{figure*}[!h]
\normalsize
\setcounter{mytempeqncnt}{\value{equation}}
\setcounter{equation}{9}
  \begin{equation}
  \label{proof-main}
\begin{aligned}
&\mathbb{E}_l \{{x}_1[kb_1+la_1]{x}^*_2[-(kb_2+la_2)]{x}_3[kb_3+la_3]\} = \\
& \sum_{i=1}^{D} A^3_i e^{j\phi_i}e^{jk\omega_{i}} + \sum_{i \not = u, v}  A_iA_uA_v e^{j[kb_1M_1\omega_i+kb_2M_2\omega_u+kb_3M_3\omega_v ]}e^{j(\phi_i-\phi_u+\phi_v)}\cdot \mathbb{E}_{l} \{e^{jl[a_1M_1(\omega_i-\omega_v)+a_2M_2(\omega_u-\omega_v)]}\}
\end{aligned}
\end{equation}
\setcounter{equation}{\value{mytempeqncnt}}
\hrulefill
\end{figure*}
\setcounter{equation}{10}
{Furthermore, if all the above assumptions hold, the time delay of the proposed scheme is determined by the maximum value of $|kb_i+la_i|$. In the following, we show the validity by giving a specific construction.}
\begin{thm}
 When $a_1M_1(\omega_i-\omega_v)+a_2M_2(\omega_u-\omega_v) \not = 0$ for $i\not =v \not = u \in \{1,2,...,D\}$, there exist constants $c_1$ and $c_2$ such that the time delay of the proposed scheme is upper bounded by $(c_1K+c_2L)MT_s$, where $K$ is the number of consecutive lags and $L$ is the number of snapshots required to estimate  $\mathbb{E}_l \{{x}_1[kb_1+la_1]{x}^*_2[-(kb_2+la_2)]{x}_3[kb_3+la_3]\}$. Here $M = \max\{M_1, M_2, M_3\}$.
\end{thm}
\begin{proof}
Consider a set of integers, say, $\{2,3,5\}$, which satisfies $gcd(2,3,5)=1$. Now we try to find out two integer groups $\{a_1,a_2,a_3\}$ and  $\{b_1,b_2,b_3\}$ such that
\begin{equation}
\left\{
            \begin{array}{lr}
2a_1+3a_2+5a_3 = 0 \\
a_1+a_2+a_3=0  \\
2b_1+3b_2+5b_3 = 1 \\
b_1+b_2+b_3=0  \\
 \end{array}
             \right.
\end{equation}

We choose $\{a_1=2,b_1=1\}$, $\{a_2=-3,b_2=-2\}$, $\{a_3=1,b_3=1\}$ as solutions of (11). Due to $\sum_i a_i=0$ and $\sum_i b_i=0$, clearly they are also solutions to
\begin{equation}
(kb_1+la_1)(2+\Gamma) + (kb_2+la_2)(3+\Gamma)+(kb_3+la_3)(5+\Gamma) =k
\end{equation}
for any $\Gamma \in \mathbb{Z}$. Based on (12), we can get the lags to estimate autocorrelation. Because $k \in \{1, 2, ..., K\}$ and $l \in \{ 1,2,...,L\}$, we have $\max_{i,k,l} |kb_i+la_i| \leq 2K+3L$.  Also $(kb_1+la_1)$ and $(kb_3+la_3)$, i.e., $(k+2l)$ and $(k+l)$, are always positive while $(kb_2+la_2)$, i.e., $(-2k-3l)$, is negative. Thus, the total time delay is upper bounded by $\max_{i,k,l} |kb_i+la_i| \cdot \max_{i}M_iT_s \leq (2K+3L)(5+\Gamma)T_s$.
\end{proof}

\noindent{ In the following, we give a general framework of multiple samplers. Given $N$ samplers, of which the sampling rates are denoted by $M_1, M_2, ... ,M_{N}$, a distributed co-prime sampling can be naturally constructed by selecting any three of them and implementing with the above scheme.} To provide a concrete strategy to efficiently make use of the triple cross difference between samples collected from each sampler, let us revisit the idea we apply before. For a subgroup of triple samplers, say, $\{M_{i_1}, M_{i_2}, M_{i_3}\}$, $i_1, i_2, i_3 \in \{1,2,...,N\}$, we still try to construct two specific solutions, $\{a_{i_1}, a_{i_2}, a_{i_3}\}$ and $\{b_{i_1}, b_{i_2}, b_{i_3}\}$, such that
\begin{equation}
\label{wanted}
\left\{
            \begin{array}{lr}
a_{i_1}M_{i_1}+a_{i_2}M_{i_2}+a_{i_3}M_{i_3} = 0 \\
a_{i_1}+ a_{i_2}+ a_{i_3}=0 \\
b_{i_1}M_{i_1}+b_{i_2}M_{i_2}+b_{i_3}M_{i_3}=1 \\
b_{i_1}+ b_{i_2}+ b_{i_3}=0 \\
\end{array}
             \right.
\end{equation}
which can be simplified to find out $a_{i_1},a_{i_3},b_{i_1},b_{i_3}$,
\begin{equation}
\label{feiyang}
\left\{
            \begin{array}{lr}
a_{i_1}(M_{i_1}-M_{i_2})+a_{i_3}(M_{i_3}-M_{i_2}) = 0 \\
b_{i_1}(M_{i_1}-M_{i_2})+b_{i_3}(M_{i_3}-M_{i_2}) = 1  \\
 \end{array}
             \right.
\end{equation}
In the following, we try to figure out how many such subgroup can be constructed from ${M_1, M_2, ... ,M_N}$. Without loss of generality, we set ${M_1, M_2, ... ,M_N}$ as the sequence of consecutive numbers starting from $1$ to $N$ shifted by some integer $\Gamma$, i.e, $M_i = i+\Gamma$. To lighten the expression, the following results are presented in an asymptotic sense of $N$. For any $M_i$, we consider the following sequence
\begin{equation}
\label{seq}
M_1-M_i, M_2-M_i, ... ,M_N-M_i
\end{equation}
and try to estimate the number of co-prime pairs among them. Since the sequence in (\ref{seq}) are still consecutive numbers ranging from $(-N,N)$, the number of primes among the sequence is upper bounded by $\pi(N)$, where $\pi(N)$ denotes the number of primes no bigger than $N$. Thus, by randomly picking any two of them, the probability of the two picked numbers which are co-prime is lower bounded by
\begin{equation}
\prod_{j=1}^{\pi(N)} (1-\frac{1}{p^2_j}) > \prod_{j=1}^{\infty} (1-\frac{1}{p^2_j}) = \frac{6}{\pi^2}
\end{equation}
where $p_j$ is the $j^{th}$ prime in the natural order and the above inequality follows from the density of primes \cite{hardy}. Here we use the fact that if we randomly select two numbers from $\mathbb{Z}$, the probability that they both share a prime factor $p_j$ is $\frac{1}{p^2_j}$. Therefore, we can totally find $\frac{6}{\pi^2}\binom{N}{3}$, i.e., $\frac{1}{\pi^2}N(N-1)(N-2)$, triplet sets such that $(M_{i_1}-M_{i_2})$ and $(M_{i_3}-M_{i_2})$ are co-prime and thus there exist solutions satisfying (\ref{feiyang}), which is equivalent to that there are solutions to (\ref{wanted}). Furthermore, without loss of generality, we assume $M_{i_1}>M_{i_2}>M_{i_3}$ and therefore, $M_{i_1}-M_{i_2}>0$ while $M_{i_3}-M_{i_2}<0$ in (\ref{feiyang}). Hence, we can specifically set $a_{i_1}=M_{i_2}-M_{i_3}$, $a_{i_3}=M_{i_1}-M_{i_2}$, $b_{i_1}=\langle (M_{i_1}-M_{i_2})^{-1} \rangle_{(M_{i_2}-M_{i_3})}$ and $b_{i_2}=\langle (M_{i_3}-M_{i_2})^{-1} \rangle_{(M_{i_1}-M_{i_2})}$, which are all positive. Thus both $a_{i_2}$ and $b_{i_2}$ should be negative due to the restrictions in (\ref{wanted}). Therefore, from (\ref{feiyang}), it is clear that both $|a_i|$ and $|b_i|$, if exist, are upper bounded by $2(N-1)$. Moreover, for $k=1,2,...,K$ and $l=1,2,...,L$, $(kb_{i_2}+la_{i_2})$ are always negative whereas $(kb_{i_1}+la_{i_1})$ and $(kb_{i_3}+la_{i_3})$ are positive. According to Theorem 1 and equation (\ref{wanted}), the following theorem can be derived.
\begin{thm}
For arbitrary $N$ samplers, there exists a distributed co-prime sampling scheme which can provide at least $L\frac{1}{\pi^2}N(N-1)(N-2)$ virtual samples to estimate $\mathbb{E}_l \{{x}_{i_1}[kb_{i_1}+la_{i_1}]{x}^*_{i_2}[-(kb_{i_2}+la_{i_2})]{x}_{i_3}[kb_{i_3}+la_{i_3}]\}$ with time delay upper bounded by $2(N-1)(K+L)MT_s$. Here $M = \max_i\{M_i\}$
\end{thm}

It is worthwhile to mention that following our idea, for any given $N$ samplers, the number of subsets of size three, where $\{a_i\}$ and $\{b_i\}$ satisfying (\ref{wanted}) exist, is upper bounded by
\begin{equation}
\binom{N}{3} - \binom{N_e}{3} - \binom{N-N_e}{3} < \frac{N(N-1)(N-2)}{8}
\end{equation}
where $N_e$ is the number of even numbers among $N$ integers. It is noted that if $M_1, M_2$ and $M_3$ are all even or odd integers, (\ref{feiyang}) is not solvable. Therefore, the proposed construction is close to the optimal.

\section{Direction of Arrival Estimation with Multiple Coprime Arrays}
\noindent As mentioned earlier, another important application of co-prime sensing is to provide enhanced freedoms for DOA estimation, which has many applications\cite{radar2,radar3,radar1}. A co-prime array structure consists of two uniform arrays with $M_1-1$ and $2M_2$ sensors respectively. {The positions of the $M_1$ sensors are given by $\{M_2m_2d, m_2=1,2,...,M_1-1 \}$ and the positions of the other $2M_2$ sensors are given by $\{ M_1m_1d, m_1=0,1,...,2M_2-1\}$.} Here $d=\frac{\lambda}{2}$ and $\lambda$ corresponds to the wavelength. As indicated by Bazout theorem, the difference set $\{m_1M_1 - m_2M_2\}$ will cover all consecutive integers from $-M_1M_2$ to $M_1M_2$, which can further provide $(2M_1M_2+1)$ DOF. In general, there are two primary concerns in DOA estimation. The first is the number of consecutive lags to estimate autocorrelation. As shown in \cite{coprime1}, both the resolution and DOF are proportional to the number of the longest consecutive lags generated by the difference set \footnote{Though consecutiveness can be relaxed by only requiring distinct lags with sparse sensing techniques \cite{sparse}. }. Second, a larger minimal distance between sensors will always be desirable in order reduce coupling.

 Let $a_l(\theta_i) = e^{j(2\pi/\lambda) d_l \sin\theta_i}$ be the element of the steering vector corresponding to direction $\theta_i$, where $d_l$ is the position of $l^{th}$ sensor. Assuming $f_c$ to be the center frequency of the band of interest, for narrow-band sources centered at $f_i+f_c$, $i=1,2,...,D$, the received signal being down-converted to baseband at the $l^{th}$ sensor is expressed by
\begin{equation}
x_l(t) = \sum_{i=1}^D a_l(\theta_i)s_i(t)e^{2\pi f_i t}
\end{equation}
where $s_i(t)$ is a narrow-band source. When a slow-fading channel is considered, we assume $s_i(t)$ as some constant $s_i$ in a coherence time block \cite{FSF}. With the similar idea we use in frequency estimation, when $n_1+n_3 = n_2$,
\begin{equation}
\label{three-chanel}
\mathbb{E}_{n_1,n_3}\big[ x_{l_1}[n_1]\cdot x^*_{l_2}[n_2] \cdot x_{l_3}[n_3] \big] = \sum_{i=1}^{D} s^3_ie^{j(2\pi/\lambda) (d_{l_1}-d_{l_2}+d_{l_3} )\sin\theta_i}
\end{equation}
where $ x_{l}[n]= x_{l}(n)+w_{l}(n)$, $l_1, l_2, l_3 \in \{1,2,...,N\}$ and $N=M_1+2M_2-1$. Similarly, $\mathbb{E}_{n_1,n_3}[x^*_{l_1}(n_1)\cdot x_{l_2}(n_2) \cdot x^*_{l_3}(n_3)] = \sum_{i=1}^{D} s^3_ie^{j(2\pi/\lambda) (-d_{l_1}+d_{l_2}-d_{l_3} )\sin\theta_i}$. With the above assumptions, it suffices to estimate DOA in the case of complex waveforms using the third order statistics rather than the second order one.

\begin{thm}
\label{DOA}
By assigning $M_1=qp_1$, $M_2=qp_2$ and $M_3=p_1p_2$ such that $q, p_1$ and $p_2$ are relatively co-prime positive integers, sensors are located at $\{m_1M_1, m_2M_2, m_3M_3 ~|~ m_1 \in \{0,1,...,2p_2-1\}, m_2 \in \{0,1,...,p_1-1\}, m_3 \in \{0,1,...,q-1 \}\}$ to form three uniform subarrays. Then the difference set $\{\pm (m_1M_1-m_2M_2) \pm m_3M_3\}$ will contain consecutive lags starting from $-p_1p_2q$ to $p_1p_2q$.
\end{thm}

\begin{proof}
Based on Bazout theorem, for $m_1M_1-m_2M_2 = q(m_1p_1-m_2p_2)$, where $m_1 \in \{0,1,...,2p_2-1\}$ and $m_2 \in \{0,1,...,p_1-1\}$, the difference set $\pm\{m_1p_1-m_2p_2\}$ enumerates $ \{-p_1p_2,-p_1p_2+1, ... , p_1p_2\}$. Now, applying Bazout theorem again on the two sequences, $\pm\{m_1M_1-m_2M_2\}$ and $\pm\{m_3M_3\}$, which essentially are the multiples of $q$ and $p_1p_2$, respectively, the triple difference set $\{\pm (m_1M_1-m_2M_2) \pm m_3M_3\}$ clearly covers each integer starting from $-p_1p_2q$ to $p_1p_2q$.
\end{proof}

{From Theorem \ref{DOA}, it shows that with $(p_1+2p_2+q-1)$ sensors, at least $(2p_1p_2q+1)$ DOF can be provided. Comparing to conventional co-prime array based DOA estimation, given $(M_1+2M_2-1)$ sensors, the corresponding DOF is $2M_1M_2+1 \leq (\frac{M_1+2M_2}{2})^2+1$. Moreover, the minimal distance between any two adjacent sensors is $\min \{q, p_1, p_2 \}$, since the positions of any two sensors share at least one common divisor from  $\{q, p_1, p_2 \}$. Thus, the minimal distance of sensor is $O(Nd)$.}

On the other hand, to estimate the autocorrelation at lag $k$, we will find the snapshots at time $n_1, n_2$, and $n_3$, where $n_1+n_3=n_2$, from the three uniform subarrays, respectively.{Therefore, assuming that each sensor has collected $L$ snapshots, we can find ${L(L-1)}/{2}$ samples for autocorrelation estimation for each lag by the proposed strategy, instead of $L$ samples in co-prime arrays, though the computational complexity may slightly increase to construct those samples.} Thus, those additional samples can compensate for precision downside of the third order statistics applied in proposed Diophantus arrays (\ref{three-chanel}), which is less robust than the second-order based estimation in co-prime arrays.

\section{Simulation}
\noindent We present the results of two numerical simulations in Fig. 1, which compares the performance of the proposed Diophantine equation based sparse reconstruction in the applications of frequency and DOA estimation with that of traditional Multiple Signal Classification (MUSIC).

For frequency estimation, we randomly generate $5$ and $10$ frequencies respectively and set $K=L$. The proposed method is used to estimate the frequencies with three samplers of down sampling rate $M_1=2+10^6$, $M_2=3+10^6$ and $M_3=5+10^6$. We average the root mean square error (RMSE) on 100 independent Monte-carlo runs with signal-to-noise ratio (SNR) ranged from -10 to 10dB. To evaluate the performance of the proposed method, we choose MUSIC as a baseline to estimate frequencies and the results are shown in Fig. 1(a). As expected, a small compromise in accuracy exists for proposed strategy since we use the third order statistics instead of the second one. However, the time delay of co-prime sampling is around $10^{6}$ times longer than that of ours. In the case of DOA estimation, we randomly generate $3$ and $10$ independent sources, respectively. For proposed Diophantus equation based arrays, we select $p_1=4$, $p_2=3$ and $q=5$ and thus totally $14$ sensors are used. We still use MUSIC algorithm as the baseline with $L=18$ and $L=50$ snapshots. As analyzed before, for each lag $k$, we can find $O(L^2)$ samples. The simulations are run 100 times and the averaged RMSEs are shown in Fig. 1(b). We can see that the proposed strategy in some cases is even with better performance than MUSIC. Furthermore, our method provides up to $149$ DOF compared with $57$ in a co-prime array. Also the minimal distance between sensors is $3d$, comparing to $d$ in an existing nested or co-prime array.
\begin{figure}
\label{sim}
	 \begin{minipage}[b]{0.5 \textwidth}
     \centering
   \includegraphics[width=0.95\textwidth]{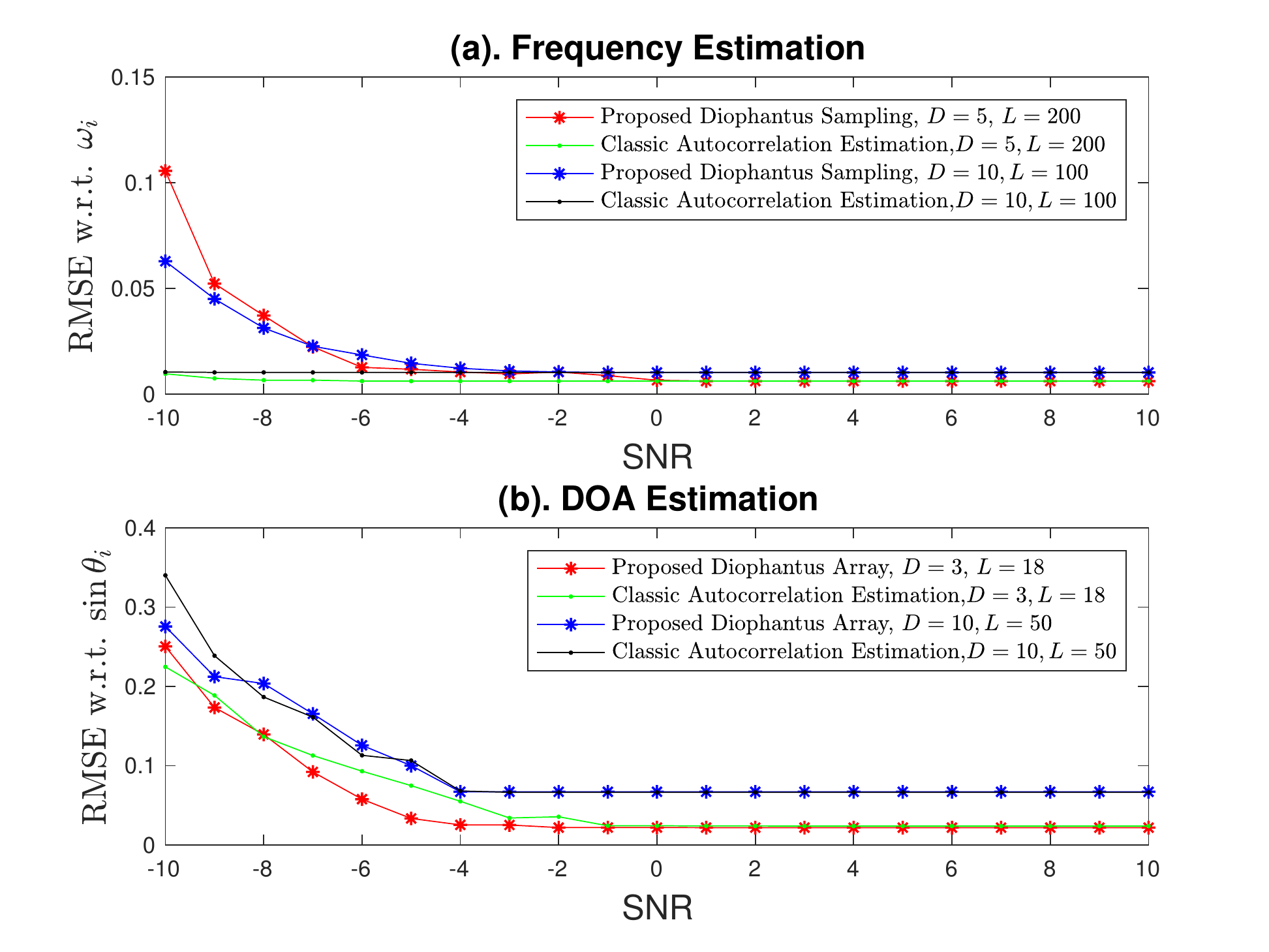}\hspace{0.01in}
   \centering \caption{Numeriacal Results Comparison}
          \end{minipage}

\end{figure}

\section{Conclusion}
\noindent In this letter, we generalize the co-prime based sparse sensing based on the idea of Diophantine equations to deal with complex waveforms. The proposed scheme establishes a new tradeoff which provides more flexibility in the parameter selection and the sparsity requirement. Experimental results also support the theory.

\bibliographystyle{plain}
\bibliography{ref}

\end{document}